\documentclass[preprint,12pt]{elsarticle}
\usepackage[utf8]{inputenc}
\usepackage[T1]{fontenc}
\usepackage{amsmath,amsthm}
\usepackage{amsfonts}
\usepackage{amssymb}
\usepackage{ textcomp }
\usepackage{tikz}
\usepackage{breqn}
\usepackage{dsfont}
\usepackage{pdfpages}
\usepackage{graphicx}
\usepackage{ upgreek }
\newtheorem{Definition}{Definition}
\newtheorem{Notation}{Notation}
\newtheorem{Theorem}{Theorem}

\begin{document}

\begin{frontmatter}
\title{Stability and Bifurcation Analysis of a Fractional Order Delay Differential Equation Involving Cubic Nonlinearity}
\author{Sachin Bhalekar\footnote{Corresponding Author Email: sachinbhalekar@uohyd.ac.in}, Deepa Gupta}
 \address{School of Mathematics and Statistics, University of Hyderabad, Hyderabad, 500046 India}
\begin{abstract}
Fractional derivative  and delay are important tools in modeling memory properties in the natural system. This work  deals with the stability analysis of a fractional order delay differential equation
\begin{equation*}
D^\alpha x(t)=\delta x(t-\tau)-\epsilon x(t-\tau)^3-px(t)^2+q x(t).
\end{equation*}
We provide linearization of this system in a neighbourhood of equilibrium points and propose linearized  stability conditions. To discuss the stability of equilibrium points, we propose various conditions on the parameters $\delta$, $\epsilon$, $p$, $q$ and $\tau$. Even though there are five parameters involved in the system, we are able to provide the stable region sketch in the $q\delta-$plane for any positive $\epsilon$ and $p$. This provides the complete analysis of stability of the system. Further, we investigate chaos in the proposed model. This system exhibits chaos for a wide range of delay parameter.
\end{abstract}
\end{frontmatter}
\section{Introduction}
The non-local operators and the delay are the crucial tools in modeling memory properties in the natural system \cite{baleanu2015chaos,bhalekar2010fractional,bhalekar2011fractional1}. The non-local operator viz. fractional order derivative(FD) is widely analyzed and applied by Scientists and Engineers \cite{debnath2003recent,lai2016investigation}. The flexible order (integer, real, complex numbers as well as functions) is yet another reason to employ the FD in the systems which show an intermediate behaviour e.g. viscoelasticity \cite{torvik1984appearance}.\\
Fractional calculus(FC) is used to model diffusion by Mainardi \cite{mainardi1996fractional,mainardi1996fundamental}, Wyss \cite{wyss1986fractional}, Luchko \cite{luchko2009maximum}, Daftardar- Gejji and coworkers \cite{jafari2006solving,daftardar2008solving}.\\
Magin \cite{magin2004fractional,magin2010fractional} presented ample number of applications of FC in bioengineering and related areas.\\
The FD is proved useful in designing robust controllers and other engineering applications \cite{lai2016investigation,monje2010fractional,chen2009fractional,yang2016fractional,chen2022fractional}.\\
Existence and uniqueness of solutions of fractional differential equations is discussed in \cite{delbosco1996existence,zhou2009existence,Daftardar2004analysis}.\\
Matignon proposed stability results of FDEs in his seminal work \cite{matignon1996stability}.\\
Since the past values of state are also included in the model, the delay differential equations become an infinite dimensional dynamical systems \cite{smith2011introduction,lakshmanan2011dynamics}.\\
The delay models are observed in various phenomena \cite{smith2011introduction,namajunas1995stabilization,bocharov2000numerical}.\\
The fractional order delay differential equations(FDDE) contain FD as well as the delay. The stability analysis of FDDE is presented by Bhalekar in \cite{bhalekar2016stability,daftardar2014new,bhalekar2013stability}.\\
Exact and discritized  stability of linear FDDEs is discussed in \cite{kaslik2012analytical} by  Kaslik and Sivasundaram.\\
Stabilization problem of neutral FDDEs is given in \cite{bonnet2007stabilization}. Various numerical schemes \cite{daftardar2015solving,bhalekar2011predictor,shi2020new,yuttanan2021legendre} are designed  by the researchers.\\
Some issues related with the initialization of FDDEs are examined in \cite{garrappa2020initial}.\\
In this work, we propose the stability results of a FDDE involving a cubic nonlinearity. The Section \ref{sec1.1} deals with the preliminaries. Stability results are proposed in Section \ref{sec1.2}. Section \ref{sec1.5} provide stable region for an equilibrium point. We analyze the chaos in proposed system in the Section \ref{sec1.3}. Finally, conclusions are given in Section \ref{sec1.4}.

\section{Preliminaries}\label{sec1.1}
In this section, we present definitions available in the literature\cite{smith2011introduction,lakshmanan2011dynamics,podlubny1998fractional,vukic2003nonlinear,diethelm2002analysis}.

\begin{Definition}[Fractional Integral]
For any $f \in \mathcal{L}^{1}(0,b)$ the Riemann-Liouville fractional integral of order $\upmu >0$, is given by 

\begin{equation*}
\textit{I}^\upmu f(t)=\dfrac{1}{\Gamma(\upmu)}\int_{0}^{t}(t-\tau)^{\upmu-1}f(\tau)d\tau  , \quad   0<t<b.
\end{equation*}

\end{Definition}

\begin{Definition}[Caputo Fractional Derivative]
For $f \in \mathcal{L}^{1}(0,b)$, $0<t<b$ and $m-1<\upmu\leq m$, $m \in \mathbb{N}$, the Caputo fractional derivative of function $f$ of order $\upmu$ is defined by,
\[\textit{D}^{\upmu} f(t)=
\begin{cases}
\frac{d^m}{dt^m} f(t) ,\textit{ if } \quad \upmu = m \\ \textit{I}^{m-\upmu}\dfrac{d^m f(t)}{dt^m}, \textit{ if } \quad  m-1< \upmu < m. 
\end{cases}\]
 Note that for  $m-1 < \upmu \leq m$, $m\in \mathbb{N},$

\[\textit{I}^\upmu\textit{D}^\upmu f(t)=f(t)-\sum_{k=0}^{m-1}\dfrac{d^k f(0)}{dt^k}\dfrac{t^k}{k!}.\]

\end{Definition}
\begin{Definition}[Equilibrium Point]
 Consider the generalized delay differential equation
  \begin{equation}
  D^\alpha{x}(t)=f(x(t),x(t-\tau)),  \quad 0<\alpha\leq 1,\label{aa}
  \end{equation}
  where $\tau \textgreater0$,  $f: E\rightarrow \mathbb{R}$, $E\subseteq \mathbb{R}^2$ is open and $f\in C^1(E)$.

  A steady state solution of equation (\ref{aa}) is called an equilibrium point.
  
  Note that $x^*$ is an equilibrium point if and only if
  \begin{equation}
  f(x^*,x^*)=0.\label{gg}
  \end{equation}
\end{Definition}
Consider the initial-value problem for the nonautonoumous delay differential equation (\ref{aa}) with the initial data
\begin{equation}
x(t)=\phi(t), -\tau\leq t\leq 0,\textit{ where } \phi:[-\tau,0]\rightarrow \mathbb{R}. \label{cc}
\end{equation}
\begin{Notation}
The solution of delay differential equation (\ref{aa}) with initial data (\ref{cc}) is denoted by $x(t,\phi).$

The norm of $\phi$ is given by 
\begin{equation*}
||\phi||=\sup_{-\tau\leq t \leq 0}|\phi(t)|.
\end{equation*}
\end{Notation}
\begin{Definition}
An equilibrium point $x^*$ of equation (\ref{aa}) is stable if for any given $\epsilon \textgreater 0$, there exist $\delta \textgreater 0$ such that $||\phi-x^*||\textless \delta \Rightarrow |x(t,\phi)-x_*|\textless \epsilon, \quad t\geq 0.$ \\
\end{Definition}

\begin{Definition}
An equilibrium point $x^*$ is asymptotically stable if it is stable and there exists $b_0\textgreater 0$ such that $||\phi-x^*||\textless b_0 \Rightarrow \lim_{t\longrightarrow\infty}x(t,\phi)=x^*$.
\end{Definition}
\begin{Definition}
Equilibrium point which is not stable is called unstable.\\ 
\end{Definition}

\begin{Theorem}\label{11}
\cite{bhalekar2016stability} Suppose $x^*$ is an equilibrium solution of the fractional order delay differential equation
\begin{equation*}
D^{\alpha} x(t)=ax(t)+b x(t-\tau). \label{12}
\end{equation*}
\textbf{Case 1} If b $\in (-\infty,-|a|)$ then the stability region of $x^*$ in $(\tau,a,b)$ parameter space is located between the plane $\tau=0$ and
\begin{equation}\label{zzzzz}
 \tau_*=\dfrac{\arccos\Bigg(\dfrac{\Bigg(acos\Big(\dfrac{\alpha\pi}{2}\Big)+\sqrt{b^2-a^2\sin^2\Big(\dfrac{\alpha\pi}{2}}\Big)\Bigg)\cos\dfrac{\alpha\pi}{2}-a}{b}\Bigg)}{\Bigg(a\cos\Big(\dfrac{\alpha\pi}{2}\Big)\pm\sqrt{b^2-a^2\sin^2\Big(\dfrac{\alpha\pi}{2}\Big)}\Bigg)^{1/\alpha}}.
 \end{equation}
The equation undergoes Hopf bifurcation at this value.\\
 \textbf{Case 2} If $b\in (-a,\infty)$ then $x^*$ is unstable for any $\tau\geq 0.$\\
 \textbf{Case 3} If $b\in (a,-a)$ and $a<0$ then $x^*$ is stable for any $\tau\geq 0.$
 \end{Theorem}
 \vspace{0.5cm}
 \textbf{ Note: } In Case 1, we say that $x^*$ is delay dependent stable.
\subsection{Linearization near equilibrium \cite{bhalekar2016stability}} 
Let $x(t)$ be a solution of the generalized fractional delay differential equation (\ref{aa}) perturbed infinitesimally from the equilibrium solution. Let $\xi(t)=x(t)-x^*$. Then by using first order Taylor's approximation, we get a linearized equation of (\ref{aa}) as
\begin{gather*}
D^{\alpha}\xi = D^\alpha x\\
           \hspace{1.5cm}= f(x(t),x(t-\tau))\\
           \hspace{2.7cm}= f(x^*+\xi(t),x^*+\xi(t-\tau))\\
           \hspace{4.6cm}= f(x^*,x^*)+\partial_1 f(x^*,x^*)\xi(t)+\partial_2 f(x^*,x^*) \xi (t-\tau).
           \end{gather*} 
           \begin{equation}
         \therefore D^\alpha\xi  = a\xi+ b\xi(t-\tau),\label{yy}
         \end{equation}
where $a=\partial_1 f$, b=$\partial_2 f$ are partial derivative of $f$ with respect to the first and second variables evaluated at $(x^*,x^*)$, respectively.
Equation (\ref{yy}) is local linearization of equation (\ref{aa}) near $x^*$. The trajectories of the generalized fractional order delay differential equation (\ref{aa}) in the neighbourhood of an equilibrium point have the same form as the trajectories of equation (\ref{yy})
 \cite{lakshmanan2011dynamics,vukic2003nonlinear}.
\section{Main Results}\label{sec1.2}
We propose the following model
\begin{equation}
D^\alpha x(t)=\delta x(t-\tau)-\epsilon (x(t-\tau))^3-p(x(t))^2+qx(t),\label{31}
\end{equation}
where $\delta$, $\epsilon$, $p$ and $q$ are all real numbers.\\
 The non-linear terms in this equation are $(x(t-\tau))^3$ and $(x(t))^2.$\\
 If $p=0$ and $q=0$ then this is an Ucar system \cite{uccar2002prototype}. Such models occur in many physical models \cite{bhalekar2012dynamical,strogatz2018nonlinear}.\\ In this case $f(x(t),x(t-\tau))=\delta x(t-\tau)-\epsilon (x(t-\tau))^3-p(x(t))^2+qx(t).$ Therefore $a=-2px+q$ and $b=\delta-3\epsilon (x-\tau)^2$. The corresponding equilibrium points are $x_1^*=0$ and $x_{2,3}^*=\dfrac{-p\pm\sqrt{p^2+4\epsilon(\delta+q)}}{2\epsilon}.$ Note that for the existence of equilibrium points $x_{2,3}^*$ we need $p^2+4\epsilon(\delta+q)\geq 0.$    
\subsection{Stability and bifurcation analysis of equilibrium point $x_1^*$.}
For $x_1^*=0$, we have $a=q$  and $b=\delta$. \\
\begin{Theorem}\label{thm1.1}
If $\delta+q>0$, then the equilibrium point $x_1^*$ is unstable for all $\tau\geq0.$
\end{Theorem}
\begin{proof}
 For $\delta+q>0$, we have $\delta>-q$.\\Therefore, $\delta\in(-q,\infty)$ which implies that $b\in(-a,\infty)$.\\
So, by Theorem (\ref{11}) Case (2), $x_1^*$ is unstable for all $\tau\geq 0$ which completes the proof.
\end{proof}
\vspace{0.5cm}
The illustration of this Theorem is given in Figure (\ref{21}) by setting $\delta=2$, $\epsilon=1$, $p=1$, $\tau=0.5$ and $q=1$.
\begin{Theorem}\label{thm1.2} 
If $\delta+q < 0 $ and $\delta \geqslant q$, then $x_1^*$ is asymptotically stable $\forall \tau \geqslant 0$.
\end{Theorem}
\begin{proof}
 The conditions $\delta+q<0$ and $\delta\geq q$ hold only when $q$ is negative and $\delta\in(q,-q)$. Therefore by Theorem (\ref{11}) Case (3), the equilibrium point $x_1^*$ is stable for all $\tau\geq0.$ This completes the proof.
\end{proof}
This result is verified in Figure (\ref{22}) by putting  $\delta=2$, $\epsilon=1$, $p=1$, $q=-3$ and $\tau=0.8.$
\begin{Theorem}\label{thm4}
If $\delta+q < 0 $, $\delta < q $, then there exists
\begin{eqnarray}
\tau_*=
\dfrac{\arccos\Bigg(\dfrac{\Bigg(q cos\Big(\dfrac{\alpha\pi}{2}\Big)+\sqrt{\delta^2-q^2\sin^2\Big(\dfrac{\alpha\pi}{2}}\Big)\Bigg)\cos\dfrac{\alpha\pi}{2}-q}{\delta}\Bigg)}{\Bigg(q\cos\Big(\dfrac{\alpha\pi}{2}\Big)\pm\sqrt{\delta^2-q^2\sin^2\Big(\dfrac{\alpha\pi}{2}\Big)}\Bigg)^{1/\alpha}}\label{ccccccccccc}
\end{eqnarray} 
such that the equilibrium point $x_1^*$ is asymptotically stable for $0 < \tau < \tau_*$ and unstable for $\tau > \tau_*$.
\end{Theorem}
\begin{proof} If $q+\delta<0$ and $\delta<q$ then $\delta$ is always negative. Therefore,  $\delta\in (-\infty,-|q|)$ and hence the condition in the Case (1) of Theorem (\ref{11}) is satisfied. Therefore there exists the critical value of delay $\tau_*$ given by the equation (\ref{ccccccccccc})
and $x_1^*$ is delay dependent stable.
\end{proof}
\vspace{0.5cm}
To verify this result, we take $\delta=-3,\quad\epsilon=1,\quad p=1\quad and \quad q=-2$ which satisfy the conditions given in the Theorem \ref{thm4}. In this case, the critical value of delay is $\tau_*\approx 1.0690$. So, for $\tau=0.8$, we get stable solution (cf. Figure (\ref{24})) whereas for $\tau=1.1$ we get unstable solution (cf. Figure (\ref{23})).\vspace{0.5cm}

\textbf{Note:} The stability of equilibrium point $x_1^*$ is independent of  values of the parameters $p$ and $q.$ \\
We also summarise the Theorems \ref{thm1.1}, \ref{thm1.2} and \ref{thm4} in Figure (\ref{fig1.111}).
 
 \begin{figure}
 \centering
\includegraphics[width=0.6\textwidth]{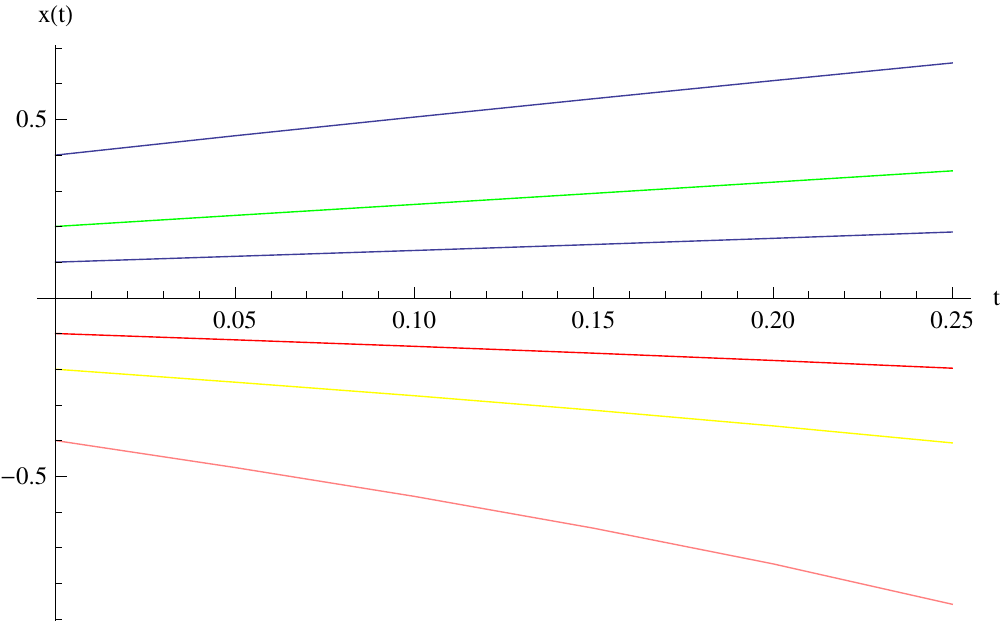}
\caption{ $x_1^*$ is unstable for $\tau=0.5$, $\delta=2$, $\epsilon=1$, $p=1$ and $q=1$}\label{21}
\end{figure}
\begin{figure}
 \centering
\includegraphics[width=0.6\textwidth]{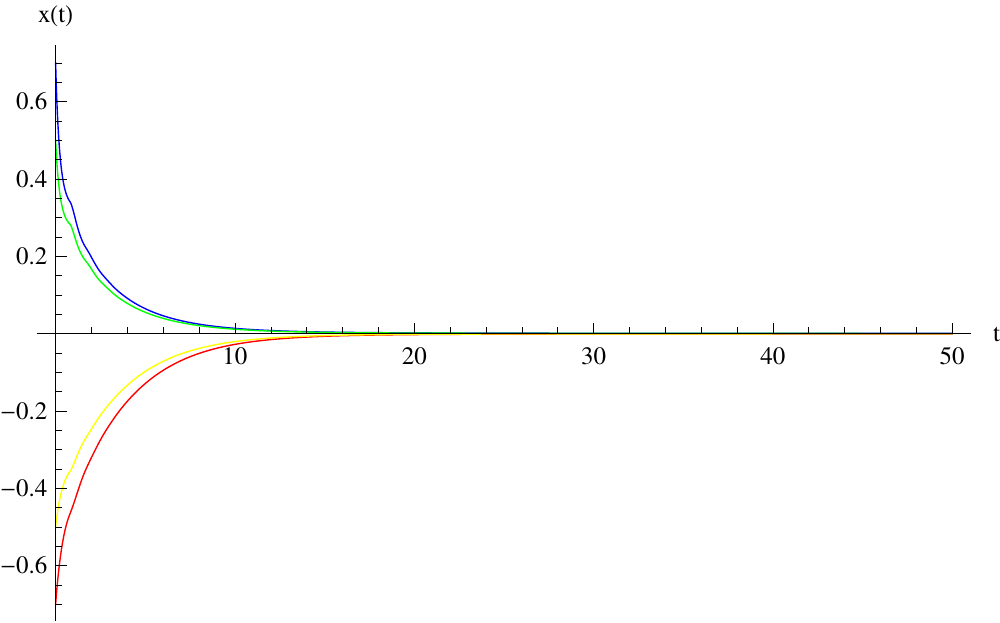}
\caption{For $\tau=0.8$, $\delta=2$, $\epsilon=1$, $p=1$ and $q=-3$, $x_1^*$ is stable}\label{22}
\end{figure}
\begin{figure}
 \centering
\includegraphics[width=0.6\textwidth]{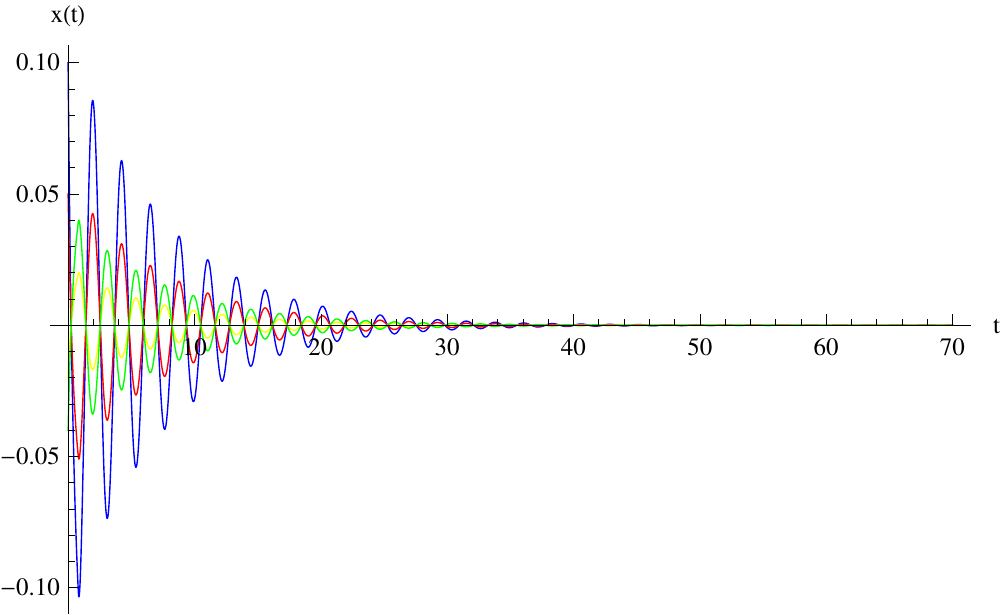}
\caption{ $x_1^*$ is stable for $\tau=0.8$, $\delta=-3$, $\epsilon=1$, $p=1$ and $q=-2$}\label{24}
\end{figure}
\begin{figure}
 \centering
\includegraphics[width=0.6\textwidth]{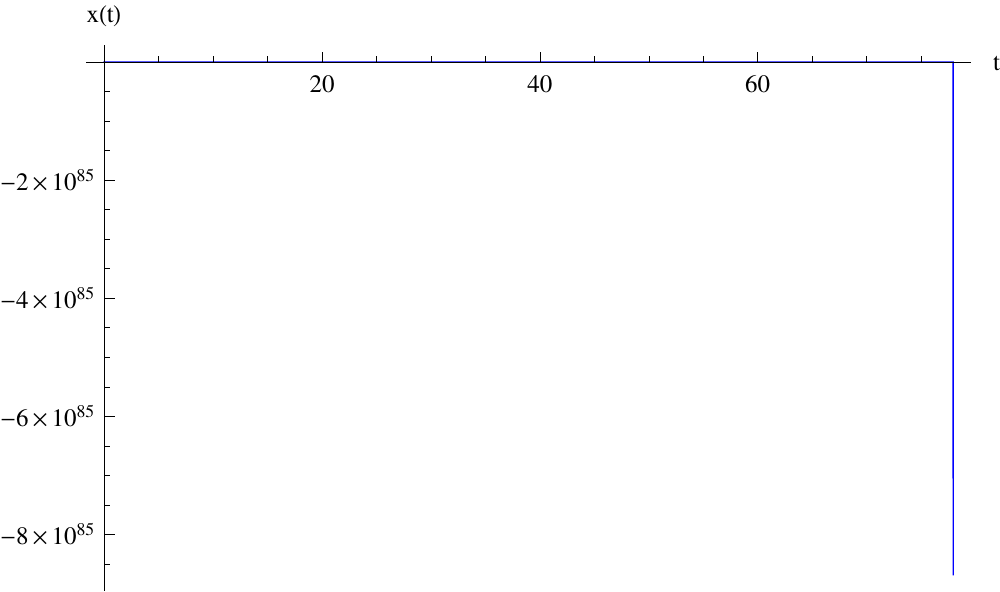}
\caption{Unstable equilibrium point $x_1^*$ for $\tau=1.1$, $\delta=-3$, $\epsilon=1$, $p=1$ and $q=-2$}\label{23}
\end{figure}
\begin{figure}
 \centering
\includegraphics[width=0.6\textwidth]{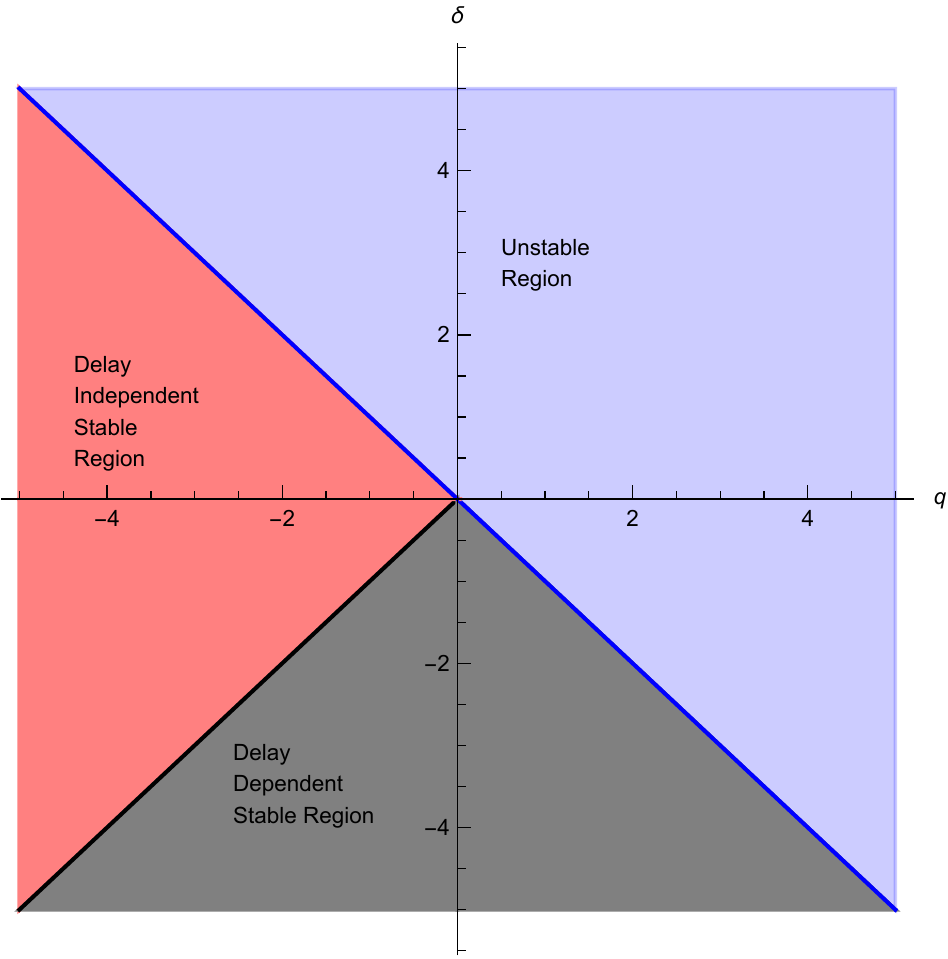}
\caption{Stability regions for $x_1^*$ in $\delta$ and $q$ plane}\label{fig1.111}
\end{figure}

\subsection{Stability and bifurcation analysis of $x_2^*$}
We have 
\begin{equation} 
	a=\dfrac{p^2}{\epsilon}-\dfrac{p\sqrt{p^2+4\epsilon(\delta+q)}}{\epsilon}+q\label{7777777777}
 \end{equation} 
and 
\begin{equation}
	b=\delta-\dfrac{3[2p^2+4\epsilon(\delta+q)-2p\sqrt{p^2+4\epsilon(\delta+q)}]}{4\epsilon}\label{9999999999}
 \end{equation}
 at $x_2^*$. Hence, we get
\begin{equation}
 a+b=\dfrac{\sqrt{p^2+4\epsilon(\delta+q)}[p-\sqrt{p^2+4\epsilon(\delta+q)}]}{2\epsilon}.\label{rrrrrrrr}
\end{equation} 
\begin{Theorem}\label{4444}
If $\epsilon>0$, $p>0$ and $0<-q<\delta<-2q$, then the equilibrium point $x_2^*$ is asymptotically stable for all $\tau\geq0$.
\end{Theorem}
\begin{proof}
Since, $\epsilon>0$ and $-q<\delta$ we have $\delta+q > 0. $\\
Also, $4\epsilon(\delta+q)>0. $\\
$\Rightarrow p^2+4\epsilon(\delta+q)> p^2$\\
$\Rightarrow \sqrt{p^2+4\epsilon(\delta+q)}-p >0 $\\
$\Rightarrow -\sqrt{p^2+4\epsilon(\delta+q)}+p < 0 $\\
$\Rightarrow \sqrt{p^2+4\epsilon(\delta+q)}[-\sqrt{p^2+4\epsilon(\delta+q)}+p] <0. $\\
So, from equation (\ref{rrrrrrrr}) 
and $\epsilon>0$ we have $a+b<0$.\\
Further, we get $a-b=4q+2\delta+\dfrac{5p^2}{2\epsilon}-\dfrac{5p\sqrt{p^2+4\epsilon(\delta+q)}}{2\epsilon}.$\\
Using, $\delta<-2q$ and $\sqrt{p^2+4\epsilon(\delta+q)}>p$, we have $ a-b< 0.$\\
$\Rightarrow a < b.$\\
 Therefore, by using Case (3) of Theorem (\ref{11}) we conclude that the equilibrium point $x_2^*$ is asymptotically stable for all $\tau\geq0$.
\end{proof}
\vspace{0.5cm}
By choosing $\delta=3$, $\epsilon=1$, $p=1$ and $q=-2$ in the equation (\ref{31}) we get stable solution for all $\tau\geq0$ by the Theorem (\ref{4444}). We verified this result by taking $\tau=\left\{0.5k|k=1,2,\ldots,20\right\}$. Figure (\ref{41}) shows stable orbit at $\tau=0.5$.
 \begin{figure}
 \centering
\includegraphics[width=0.6\textwidth]{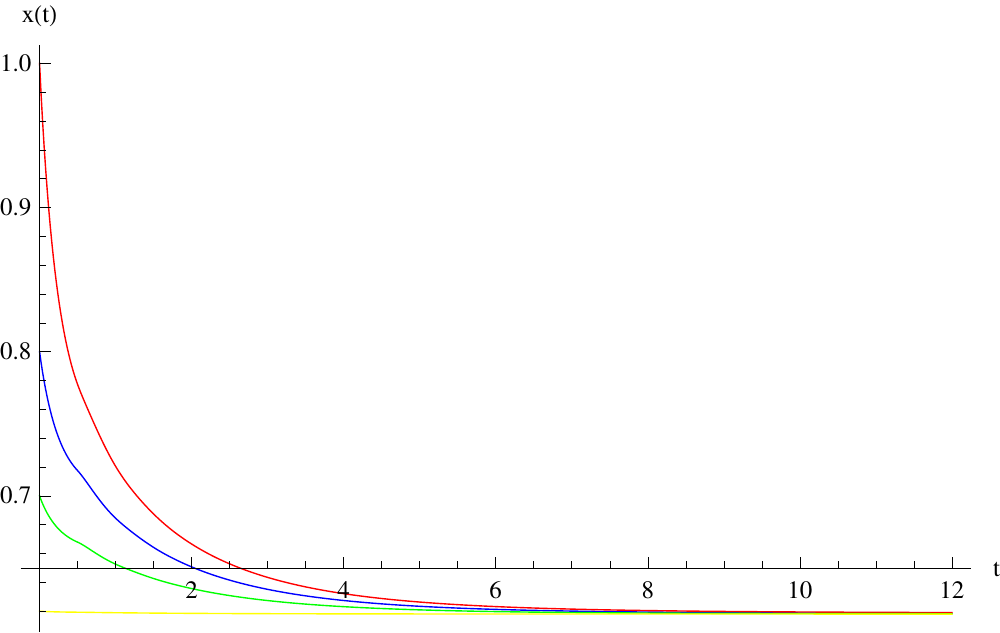}
\caption{ $x_2^*$ is stable for $\tau=0.5$, $\delta=3$, $\epsilon=1$, $p=1$ and $q=-2$}\label{41}
\end{figure}
\begin{Theorem}\label{6666}
If $\epsilon>0$, $p>0$, $\delta<\dfrac{-p^2}{32\epsilon}<0$ and $(q+\delta)>0$ then there exists $\tau_*$ as given in equation (\ref{zzzzz}) (with $a$ and $b$ are defined in equations (\ref{7777777777}) and (\ref{9999999999}) respectively) such that $x_2^*$ is asymptotically stable for $0<\tau<\tau_*$ and unstable for $\tau>\tau_*.$

\end{Theorem}
\begin{proof}  \textbf{Step 1:} 
Since $q+\delta>0$, $\epsilon>0$ and $p>0$ we have, $p^2+4\epsilon(\delta+q)>p^2.$\\
$\Rightarrow p-\sqrt{p^2+4\epsilon(\delta+q)}<0.$\\ 
Therefore by equation (\ref{rrrrrrrr}), $(a+b)<0.$\\
\textbf{Step 2:} Since, $\delta<\dfrac{-p^2}{32\epsilon},\textit{ we have } q+\delta<q-\dfrac{p^2}{32\epsilon}$\\
 $\Rightarrow 4\epsilon(q+\delta)<4\epsilon(q-\dfrac{p^2}{32\epsilon})$\\
 $\Rightarrow 0< p^2+4\epsilon(\delta+q)<p^2+4\epsilon q -\dfrac{p^2}{8}\\
 \Rightarrow \sqrt{p^2+4\epsilon(\delta+q)}<\sqrt{\dfrac{7p^2}{8}+4\epsilon q}\\
 \Rightarrow \dfrac{-p}{\epsilon}\sqrt{p^2+4\epsilon(\delta+q)}>-\dfrac{p}{\epsilon}\sqrt{\dfrac{7p^2}{8}+4\epsilon q}\\
 \Rightarrow a > q+\dfrac{p^2}{\epsilon}-\dfrac{p}{\epsilon}\sqrt{\dfrac{7p^2}{8}+4\epsilon q}\\
 \Rightarrow a > -\delta+\dfrac{p^2}{\epsilon}-\dfrac{p}{\epsilon}\sqrt{\dfrac{7p^2}{8}+4\epsilon q}$.\\
  Since, $32\epsilon q >-32\epsilon\delta>p^2$\\
  we have                                                                                       $ a>-\delta+\dfrac{p^2}{\epsilon}-\dfrac{p^2}{\epsilon}$\\
  $\Rightarrow a>-\delta$\\
  $\Rightarrow a>0$.\label{ll}\\
  \textbf{Step 3:} Further, $b$ can also be written as $b=\delta-\dfrac{3\epsilon(a-q)^2}{4p^2}.$\\
  $\Rightarrow b<\dfrac{-p^2}{32\epsilon}-\dfrac{3\epsilon(a-q)^2}{4p^2}$
  which shows that $b<0.$\label{mm}\\
  So, $a+b<0$ and $b<a.$\\
  $\Rightarrow b\in(-\infty,-|a|).$\\
  Therefore, by using Case (1) of Theorem (\ref{11}), we get the required critical value $\tau_*$ of delay.
  
   \end{proof}\vspace{0.5cm}
   By setting $\delta=-1/2$, $p=4$, $\epsilon=2$ and $q=1$ we get $\tau_*\approx 2.6521$. The convergent solution for $\tau=1.8$ is given in Figure (\ref{64}), whereas divergent solution for $\tau=2.9$ is given in Figure (\ref{10}).\\
   \begin{Theorem}\label{thm1.8}
   If $\epsilon>0$, $p>0$ and $(q+\delta)<0$ then $x_2^*$ is unstable for all $\tau\geq0.$
\end{Theorem}
\begin{proof}
Since $\epsilon>0$ and $(q+\delta)<0,$\\
 we have $0<p^2+4\epsilon(\delta+q)<p^2.$\\
Therefore, $\sqrt{p^2+4\epsilon(\delta+q)}<p.$\\
So, from (\ref{rrrrrrrr}) we get $a+b>0.$\\
Hence, by Case (2) of Theorem (\ref{11}), we get the required result.
\end{proof}\vspace{0.5cm}
We verify the Theorem (\ref{thm1.8}) by putting $\delta=-2$, $\epsilon=1$, $p=3$, $q=1$ and $\tau=0.5$ in equation (\ref{31}) for which we get the unstable curve which is shown in Figure (\ref{fig1.2222222}).
 \begin{figure}
 \centering
\includegraphics[width=0.6\textwidth]{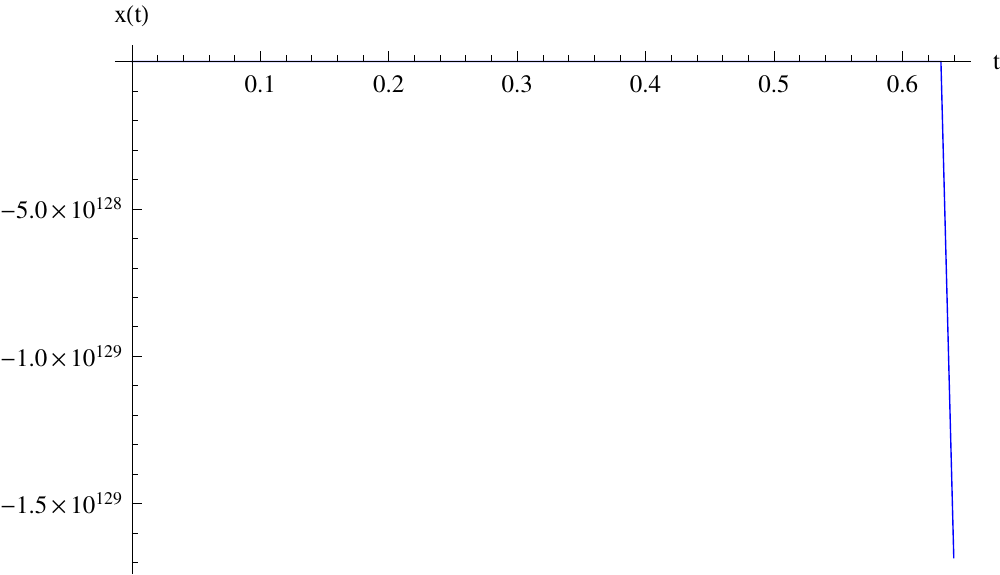}
\caption{ $x_2^*$ is unstable for $\tau=0.5$, $\delta=-2$, $\epsilon=1$, $p=3$ and $q=1$}\label{fig1.2222222}
\end{figure}

    \begin{figure}
 \centering
\includegraphics[width=0.6\textwidth]{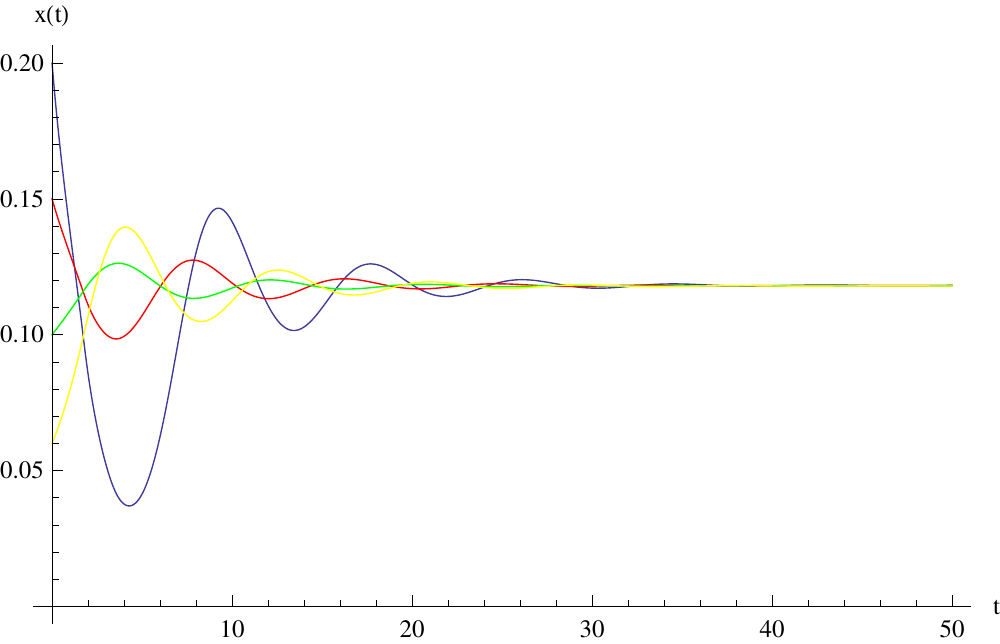}
\caption{$x_2^*$ is asymptotically stable for $\delta=-1/2$, $\epsilon=2$, $p=4$, $q=1$ and $\tau=1.8$  }\label{64}
\end{figure}
  \begin{figure}
 \centering
\includegraphics[width=0.6\textwidth]{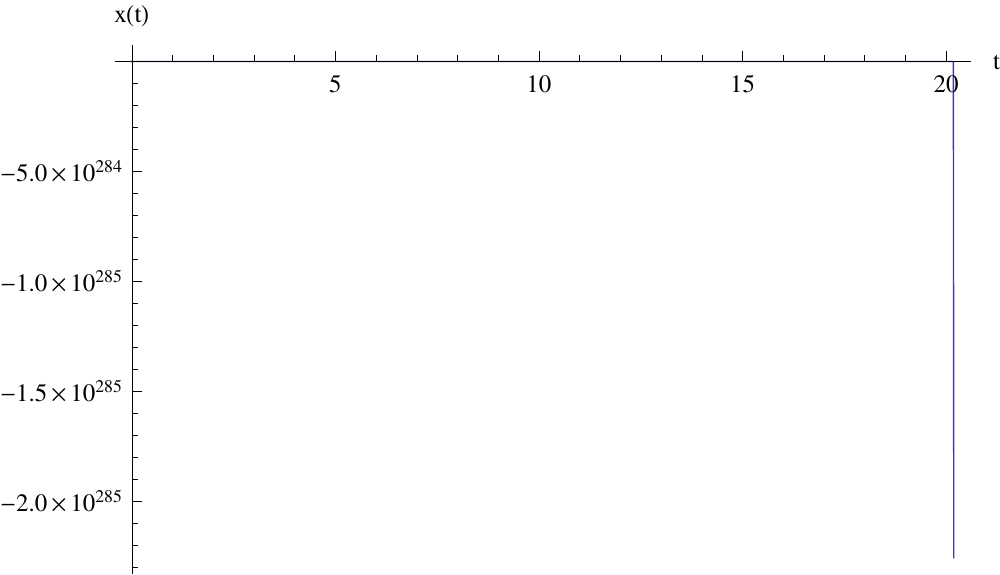}
\caption{$x_2^*$ is unstable for $\delta=-1/2$, $\epsilon=2$, $p=4$, $q=1$ and $\tau=2.9$  }\label{10}
\end{figure}
We propose few more stability results in the Theorems \ref{nn}, \ref{11111} and \ref{zz}.
  \begin{Theorem}\label{nn}
 If $\epsilon<0$, $p>0$ and $q+\delta<0$ then $x_2^*$ is unstable for all $\tau\geq0.$
  \end{Theorem}
  \begin{proof} For, $\epsilon<0$ and $q+\delta<0$ we have $4\epsilon(q+\delta)>0.$\\
  So, $p^2+4\epsilon(\delta+q)>p^2>0.$\\
  $\Rightarrow p-\sqrt{p^2+4\epsilon(\delta+q)}<0$\\
  $\Rightarrow \sqrt{p^2+4\epsilon(\delta+q)}[p-\sqrt{p^2+4\epsilon(\delta+q)}]<0$\\
  Therefore, by equation (\ref{rrrrrrrr}), we conclude that $a+b>0$ because   $\epsilon<0$.
   Hence $b\in(-a,\infty).$\\
  Therefore, we get the required result by employing Case (2) of Theorem (\ref{11}).
  \end{proof}\vspace{0.5cm}
 We verified Theorem (\ref{nn}) by setting $\delta=1$, $\epsilon=-1$, $p=1$, $q=-2$ and various values of $\tau\in(0,100)$. Figure (\ref{pp}) shows unbounded solution in this case, with $\tau=0.3.$
 \begin{figure}
 \centering
\includegraphics[width=0.6\textwidth]{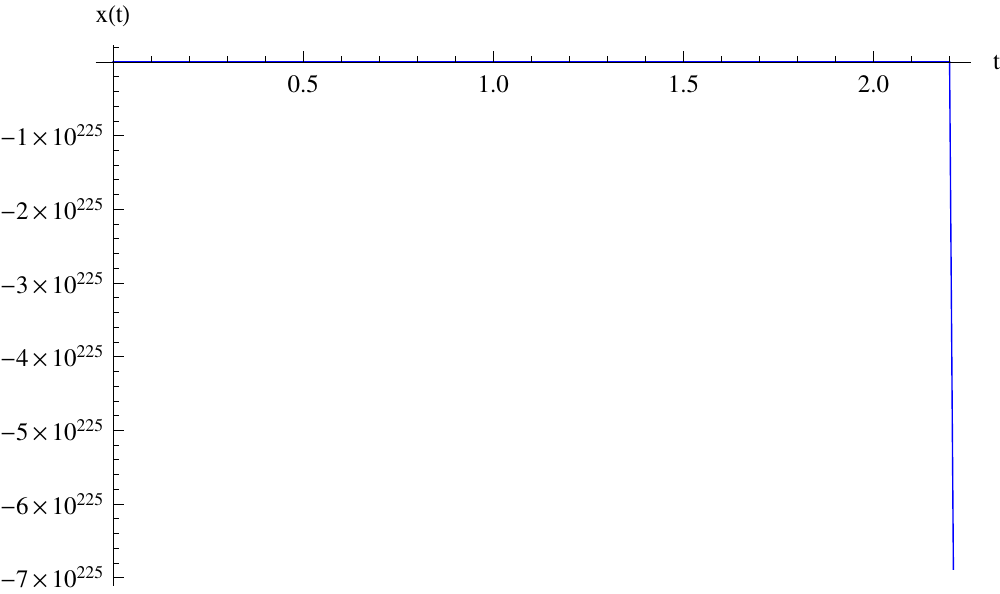}
\caption{For $\delta=1$, $\epsilon=-1$, $p=1$, $q=-2$ and $\tau=0.3$ we get unstable solution for $x_2^*$  }\label{pp}
\end{figure}
\begin{Theorem}\label{11111}
If $\epsilon>0$, $p<0$, $\delta\leq\dfrac{3p^2}{4\epsilon}$ and $q>\dfrac{-p^2-4\delta\epsilon}{4\epsilon}$ then $\exists$ $\tau_*$ as given in (\ref{zzzzz}), where $a$ and $b$ are defined by (\ref{7777777777}) and (\ref{9999999999}) respectively, such that $x_2^*$ is asymptotically stable for $0<\tau<\tau_*$ and unstable for $\tau>\tau_*.$
\end{Theorem}
\begin{proof}
	We can conclude that $a+b<0$ by using the equation (\ref{rrrrrrrr}) and the conditions $\epsilon>0$ and $p<0$.
	Therefore, to prove this Theorem we need either $b<a$ or $a>0.$\\
However, 
$ a>\dfrac{-p^2-4\delta\epsilon}{4\epsilon}+\dfrac{p^2}{\epsilon}-\dfrac{p\sqrt{p^2+4\epsilon(\delta+q)}}{\epsilon}$ because $q>\dfrac{-p^2-4\delta\epsilon}{4\epsilon}$.\\
Also, we have $\delta\leq \dfrac{3p^2}{4\epsilon}$.\\
$\Rightarrow-\delta\geq \dfrac{-3p^2}{4\epsilon}.$\\
Using these inequalities we have,
$a>\dfrac{-p\sqrt{p^2+4\epsilon(\delta+q)}}{\epsilon}.$\\
Further, $p<0$ and $\epsilon>0$ so, we have $a>0.$\\
Hence, $b\in(-\infty,-|a|)$. Therefore, the proof follows by using the Case (1) of Theorem (\ref{11}).
\end{proof}\vspace{0.5cm}
Illustration of this theorem is given in Figure (\ref{42}) by setting the parameters as $\delta=\dfrac{3}{8}$, $\epsilon=2$, $p=-1$ and  $q=1$ in equation (\ref{31}). For this set of parameters, we get $x_2^*=1.11603$ and $\tau_*=0.157185.$ Hence $\tau=0.08$ gives stable solution (cf. Figure (\ref{42})) and     
$\tau=0.36$ gives unstable solution(cf. Figure (\ref{43})).
\begin{figure}
 \centering
\includegraphics[width=0.6\textwidth]{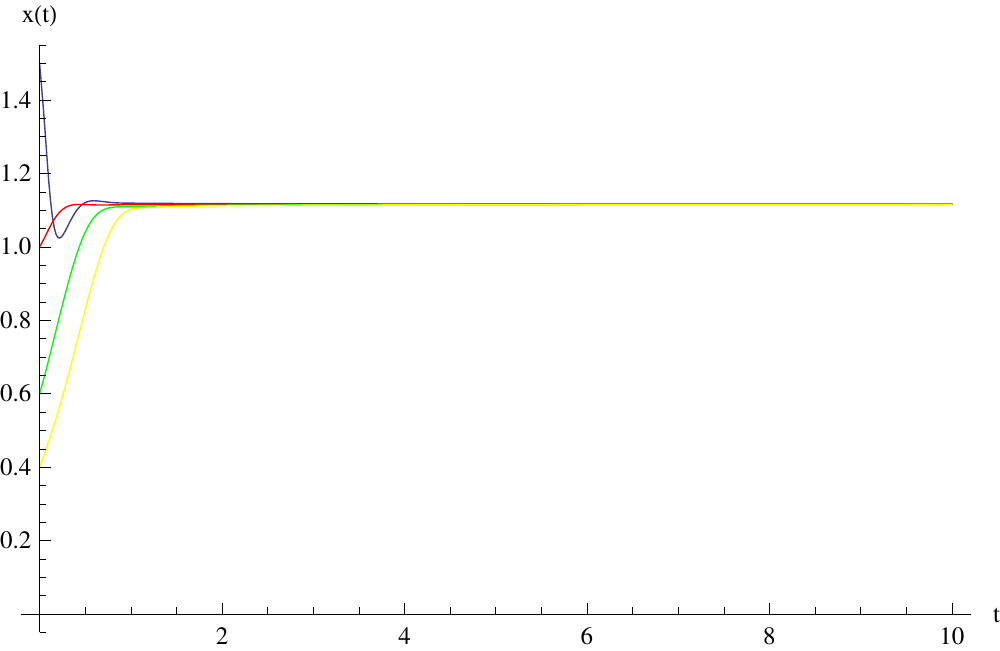}
\caption{ $x_2^*$ is stable for $\tau=0.08, \delta=3/8, \epsilon=2$,$p=-1$ and  $q=1$}\label{42}
\end{figure}
 \begin{figure}
 \centering
\includegraphics[width=0.6\textwidth]{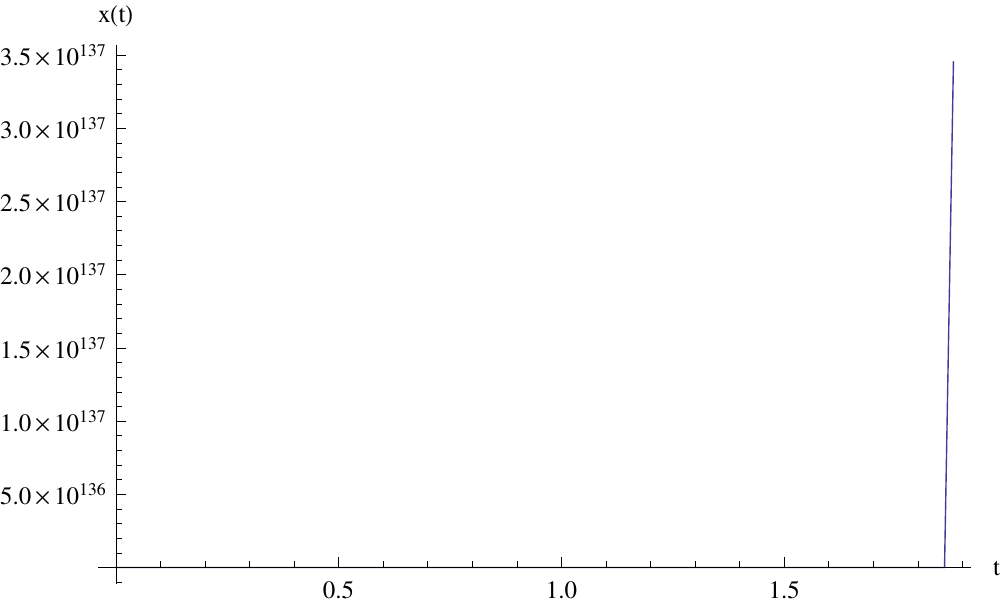}
\caption{ $x_2^*$ is unstable for $\tau=0.36, \delta=3/8, \epsilon=2$, $p=-1$ and  $q=1$}\label{43}
\end{figure}

\begin{Theorem}\label{zz}
If $\epsilon<0$, $p<0$ and $\delta+q<0$ then $x_2^*$ is unstable for all $\tau\geq 0$.
\end{Theorem}
\begin{proof} We can write $a+b$  as $\dfrac{-[4\epsilon(\delta+q)+p^2]+p\sqrt{p^2+4\epsilon(\delta+q)}}{2\epsilon}.$
\\Since $p<0$ and $\epsilon$ is also negative quantity so, $a+b>0$. Therefore by Case (2) Theorem (\ref{11}) we have the equilibrium point $x_2^*$ is unstable for $\tau\geq 0$.
\end{proof}\vspace{0.5cm}
Verification of this Theorem (\ref{zz}) is given in Figure (\ref{51})  by setting $\delta=-2$, $p=-1$, $q=1$, $\epsilon=-1$ and $\tau=0.2$.
\begin{figure}
 \centering
\includegraphics[width=0.6\textwidth]{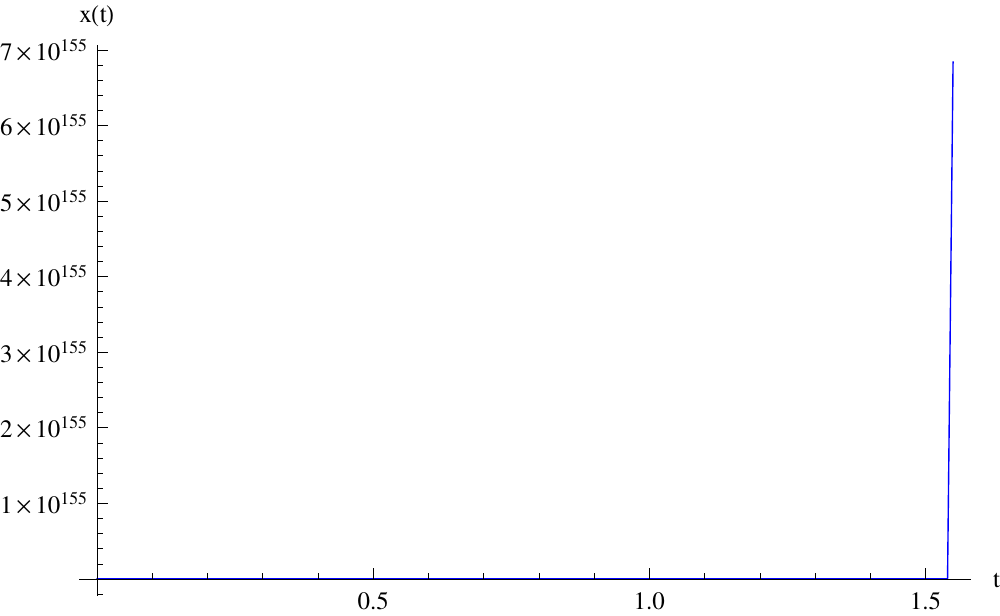}
\caption{ Theorem (\ref{zz}) is verified for the parameter values $\delta=-2$, $\epsilon=-1$, p=-1 and q=1 and $\tau=0.2$}\label{51}
\end{figure}
\section{Stable region for $x_2^*$ }\label{sec1.5}
In this section, we sketch the stable region for $x_2^*$ with $\epsilon>0$ and $p>0.$  \\
From Theorem (\ref{11}), it is clear that the curves $a-b=0$ (with $a\leq0$) and $a+b=0$ are bifurcation curves.\\
In this case, $a+b=0$ gives $\delta=-q.$\\
Further, $a-b=0$ with $a\leq 0$ gives, $\delta=g_1(p,q,\epsilon)$ and  $\delta=g_2(p,q,\epsilon),$ where $ g_1(p,q,\epsilon)=\dfrac{15p^2-16q\epsilon+5\sqrt{9p^4-16p^2q\epsilon}}{8\epsilon},$  $-\infty<q\leq q_2$,\\    
$g_2(p,q,\epsilon)=\dfrac{15p^2-16q\epsilon-5\sqrt{9p^4-16p^2q\epsilon}}{8\epsilon},$ $q_3\leq q\leq q_2$,\\ $q_2=\dfrac{9p^2}{16\epsilon}$ and $q_3=\dfrac{-p^2}{\epsilon}.$\\
The branch $\delta=g_2(p,q,\epsilon)$ will be valid for $q\in \Big[q_3,q_2\Big].$ If $q<q_3$ then either $a$ and $b$ become complex numbers or $a\neq b$, along the curve $\delta=g_2(p,q,\epsilon)$. If $q>q_2$ then $g_2(p,q,\epsilon)\notin \mathbb{R}.$\\
We have following observations:
\begin{itemize}
\item[(1)] $g_1(p,q,\epsilon)$ is decreasing in the interval $-\infty<q\leq q_2$ because
\[\dfrac{\partial g_1}{\partial q}=-2-5\dfrac{p}{\sqrt{9p^2-16q\epsilon}}<0\],\\
 for $p>0$ and $\epsilon>0$. 
\item[(2)] $\lim_{q\to q_2}g_1(p,q,\epsilon)=\lim_{q\to q_2}g_2(p,q,\epsilon)=\dfrac{3p^2}{4\epsilon}.$ Furthermore,  $\lim_{q\to -\infty}g_1(p,q,\epsilon)=\infty.$
\item[(3)] Nature of $g_2(p,q,\epsilon):$
\textbullet $g_2(p,q,\epsilon)$ is monotonically decreasing in $q_3\leq q \leq q_0$ where $q_0=\dfrac{11p^2}{64\epsilon}:$\\
We have $\dfrac{\partial g_1}{\partial q}=-2-5\dfrac{p}{\sqrt{9p^2-16q\epsilon}}$ \label{eq1}
and $q_3<q<q_0.$\\
$\Rightarrow \dfrac{25p^2}{4}<9p^4-16q\epsilon<25p^2$\\
$\Rightarrow\dfrac{1}{5p}<\dfrac{1}{\sqrt{9p^2-16q\epsilon}}<\dfrac{2}{5p}$\\
$\Rightarrow 1<\dfrac{5p}{\sqrt{9p^2-16q\epsilon}}<2$\\
$\Rightarrow \dfrac{\partial g_1}{\partial q}<0$ $\forall q\in \Big[q_3,q_0\Big]. $\\
\textbullet $g_2(p,q,\epsilon)$ is monotonic increasing in $q_0\leq q \leq q_2:$\\
If $q\in \Big[q_0,q_2\Big]$, then
$ \dfrac{11p^2}{4}<16q\epsilon<9p^2.$\\
$\Rightarrow 0<9p^2-16q\epsilon<\dfrac{25p^2}{\epsilon}$\\
$\Rightarrow 2<\dfrac{5p}{\sqrt{9p^2-16q\epsilon}}.$\\
So, $\dfrac{\partial g_2}{\partial q}>0$ $\forall q\in \Big[q_0,q_2\Big]$.\\
\textbullet Local minima of $g_2(p,q,\epsilon)$ is at $q=q_0$ with minimum value $\delta_1=\dfrac{-p^2}{32\epsilon}.$\\
Since $|q_0|>|\delta_1|$, $g_2(p,q,\epsilon)$ lies above the curve $\delta=-q$ for $0\leq q \leq q_2.$
\item[(4)] Intersection of $\delta=g_2(p,q,\epsilon)$ and $\delta=-q$ is $(0,0)$. Further, $\delta=g_2(p,q,\epsilon)$ intersects $q$ axis at $(0,0)$ and at $(q_1,0)$ where $q_1=\dfrac{5p^2}{16\epsilon}$.\\
$\lim_{q\to q_3}g_2(p,q,\epsilon)=\dfrac{3p^2}{4\epsilon}<\dfrac{p^2}{\epsilon}.$\\
So, in the interval $q_3\leq q \leq 0,$ the curve $\delta=g_2(p,q,\epsilon)$ lies below the curve $\delta=-q.$
\item[(5)] The curve $\delta=g_1(p,q,\epsilon)$ is always above the curve $\delta=-q:$\\
Suppose $q<0$.\\
Since, $g_1(p,q,\epsilon)=-2q+\dfrac{15p^2+5p\sqrt{9p^2-16q\epsilon}}{8\epsilon}$ and $p$ and $\epsilon$ both are positive, $g_1(p,q,\epsilon)>-2q>-q.$\\
Further for any $q$,\\
 $g_1(p,q,\epsilon)$ = $\dfrac{6p^2+\sqrt{9p^2-16q \epsilon}(\sqrt{9p^2-16q\epsilon}+5p)}{8\epsilon}>0$ because $p$ and $\epsilon$ both are positive.\\
Therefore, $g_1(p,q,\epsilon)>-q$, $q\in\Big(-\infty,q_2\Big).$
\item[(6)] Intersection of $\delta=g_1(p,q,\epsilon)$ with $\delta$ axis is $\delta_0=g_1(p,0,\epsilon)=\dfrac{30p^2}{8\epsilon}>0.$\\
Using these observations, we sketch  the stability regions for $x_2^*$ in Figure (\ref{fig1.5}).\\ We have-
\end{itemize}
\begin{figure}
 \centering
\includegraphics[width=0.75\textwidth]{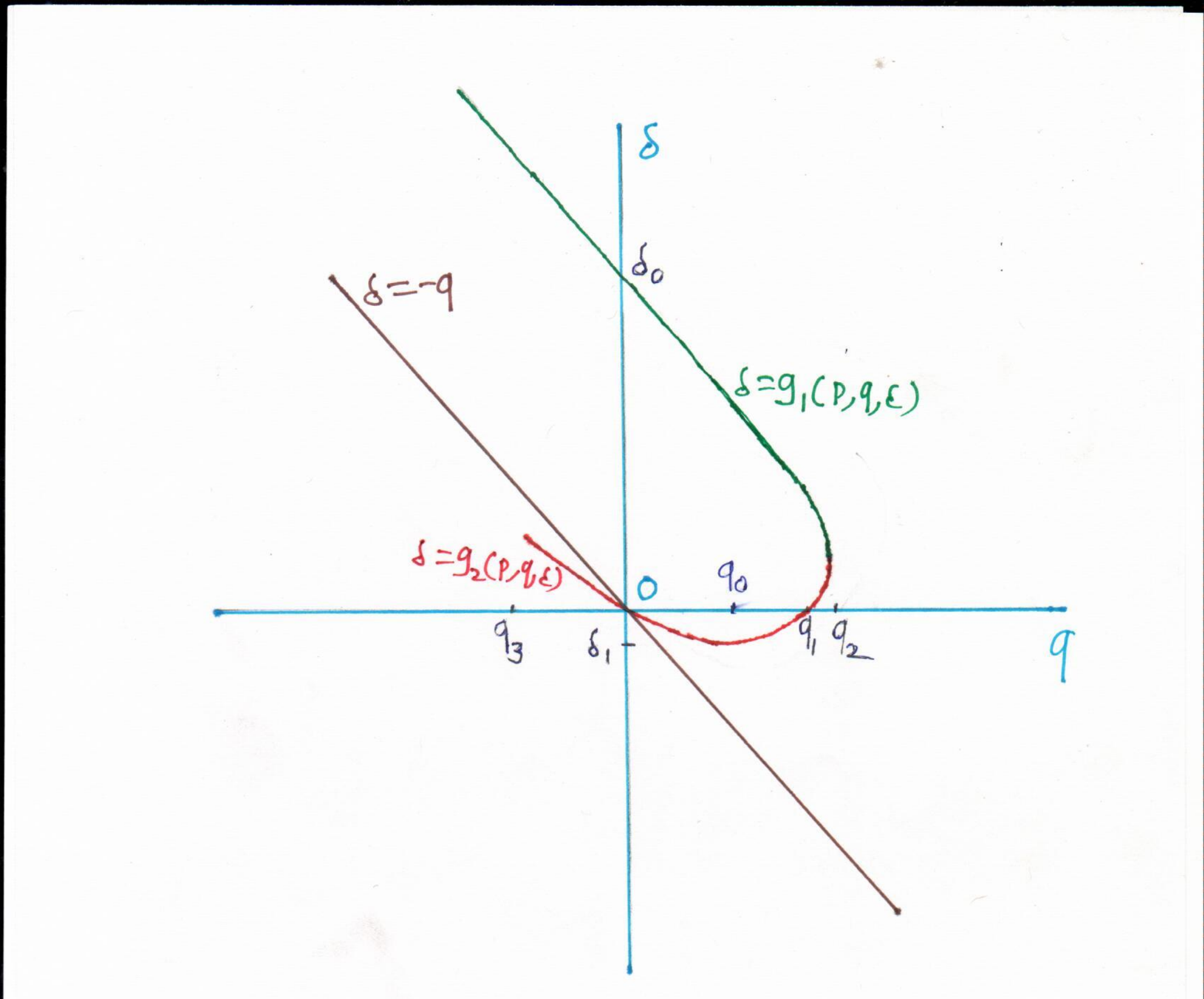}
\caption{ Stability region for $x_2^*$ in $q\delta-$plane}\label{fig1.5}
\end{figure}
 \begin{itemize}
  \item[(A)] If $q>q_1$ and $\delta>-q$, then $x_2^*$ is delay dependent stable.
  
   \item[(B)] If $-q<\delta< g_1(p,q,\epsilon)$ and $q<0$ then $x_2^*$ is asymptotically stable $\forall \tau \geqslant 0.$  
   
  \item[(C)] If $0\leq q \leq q_1$ and-\\
   \begin{itemize}
   \item[(i)] $\delta\in\Big(-q,g_2(p,q,\epsilon)\Big)\cup \Big(g_1(p,q,\epsilon),\infty\Big)$ then  $\exists$ $\tau_*$ as given in (\ref{zzzzz}) and $x_2^*$ is delay dependent stable.  \\
   \item[(ii)] $g_2(p,q,\epsilon)<\delta<g_1(p,q,\epsilon)$ then $x_2^*$ is asymptotically stable, $\forall \tau \geqslant 0$. \\
   \end{itemize}
   \end{itemize}
\section{Chaos}\label{sec1.3}
We observed chaotic oscillations in system (\ref{31}) for some parameter values. The Figure (\ref{57575757}) shows the bifurcation diagram for the parameter set $\delta=5$, $\epsilon=2$, $p=0.01$, $q=-2$ and $\alpha=0.95$. The horizontal axis is the delay $\tau$. Equilibrium points in this case are $x_1^*=0$, $x_2^*=1.22725$ and $x_3^*=-1.22725.$\\
The equilibrium point $x_1^*$ is unstable for all $\tau\geq 0$. For $x_2^*$, $b<a<0$. Therefore, there exists $\tau_*=0.6427665$ such that $x_2^*$ is asymptotically stable for $0<\tau<\tau_*$.\\
Similarly, for $x_3^*$, $b<a<0$ and $\tau_*=0.620447$. Thus, the system is unstable for $\tau>0.642766.$\\
We observed periodic limit cycles for $0.65\leq\tau\leq2.2$.\\
Figures (\ref{575757}) and (\ref{57575756}) show periodic limit cycles for $\tau=1.4$ and $\tau=1.95$ respectively.\\
Chaos is observed for $\tau>2.2$. Figures (\ref{5757575675}) and (\ref{5757575689}) show chaotic attractors for $\tau=2.3$ and $\tau=2.5$ respectively. The chaos is confirmed with bifurcation diagram (cf. Figure (\ref{57575757})) and  the positive values of maximum Lyapunov exponents (Table \ref{56565656}).
\begin{table}
\centering
\begin{tabular}{|l|p{3cm}|p{5cm}|}
\hline
$\tau$ & Maximum Lyapunov Exponents & Behaviour of System\\
\hline
0.6&-0.912265&Stable\\
\hline
1.6&-0.002104&  Limit cycle\\
\hline
1.8&-0.000428&  Limit cycle\\
\hline
2.3&0.546279 & Chaotic oscillations\\
\hline
2.5 & 1.083852 & Chaotic oscillations\\
\hline
\end{tabular}
\caption{ Maximum Lyapunov Exponents}\label{56565656}
\end{table}
We used the algorithm described by Kodba et al \cite{kodba2004detecting} which is based on the time series analysis techniques and the work by Wolf et al \cite{wolf1985determining}. 
\begin{figure}
 \centering
\includegraphics[width=0.6\textwidth]{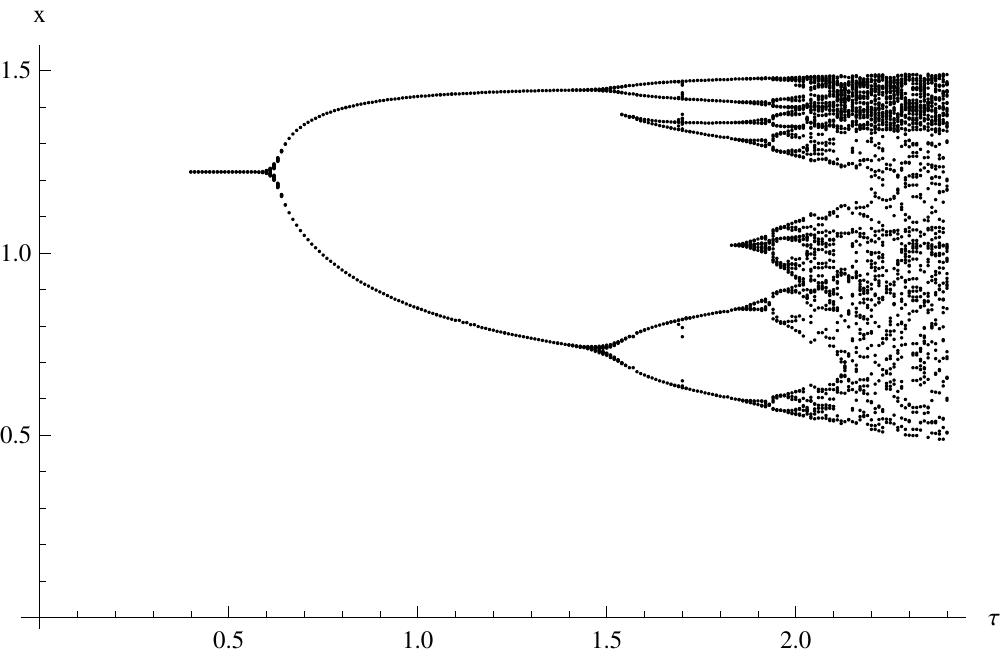}
\caption{ Bifurcation diagram}\label{57575757}
\end{figure}
\begin{figure}
 \centering
\includegraphics[width=0.6\textwidth]{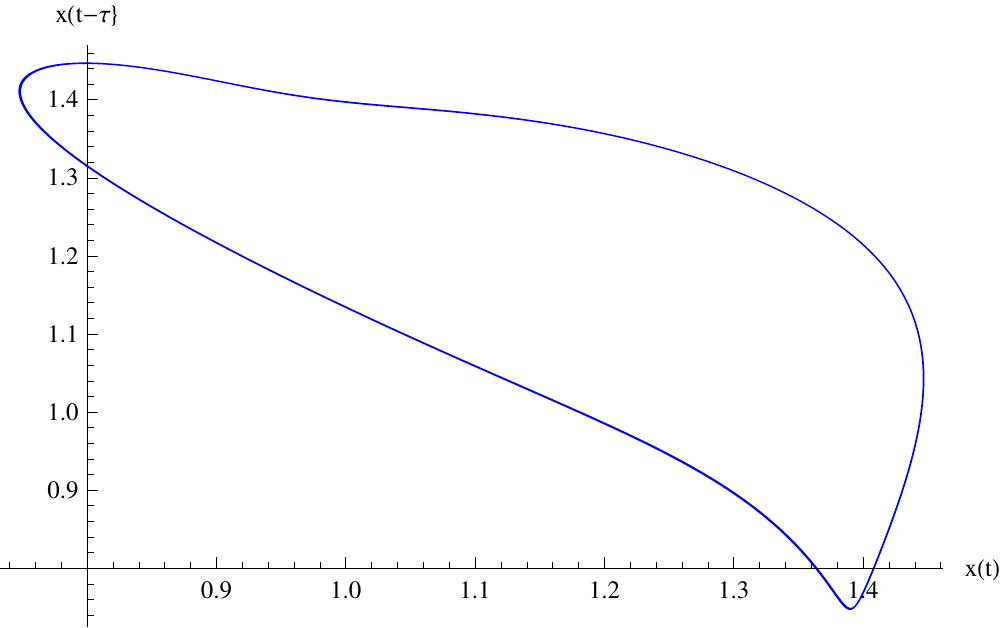}
\caption{  Graph of $x(t)$ versus $x(t-\tau)$ for  $\tau=1.4$  }\label{575757}
\end{figure}
\begin{figure}
 \centering
\includegraphics[width=0.6\textwidth]{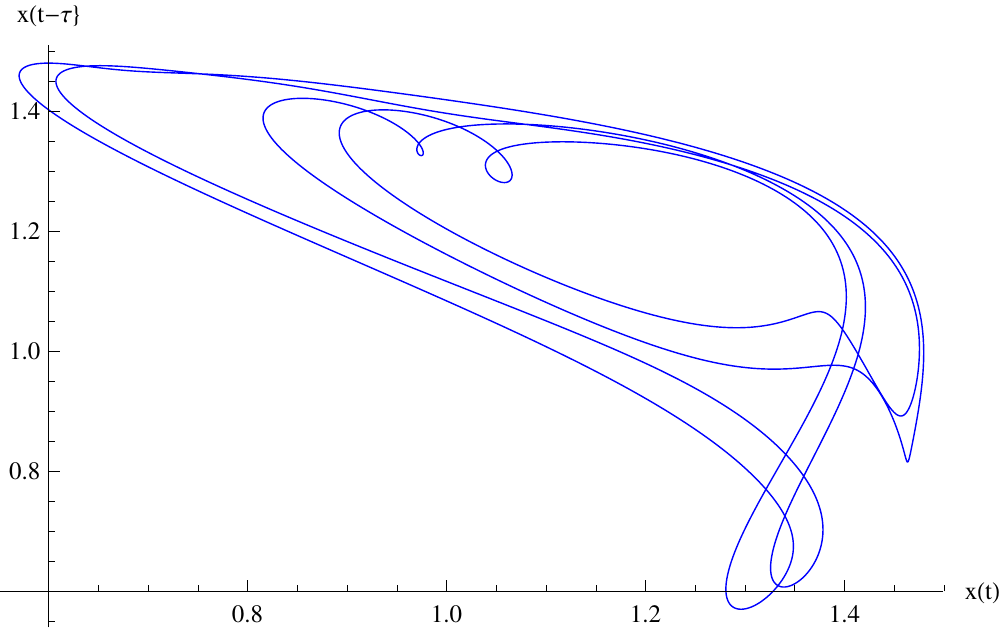}
\caption{ Periodic limit cycle for  $\tau=1.95$}\label{57575756}
\end{figure}
\begin{figure}
 \centering
\includegraphics[width=0.6\textwidth]{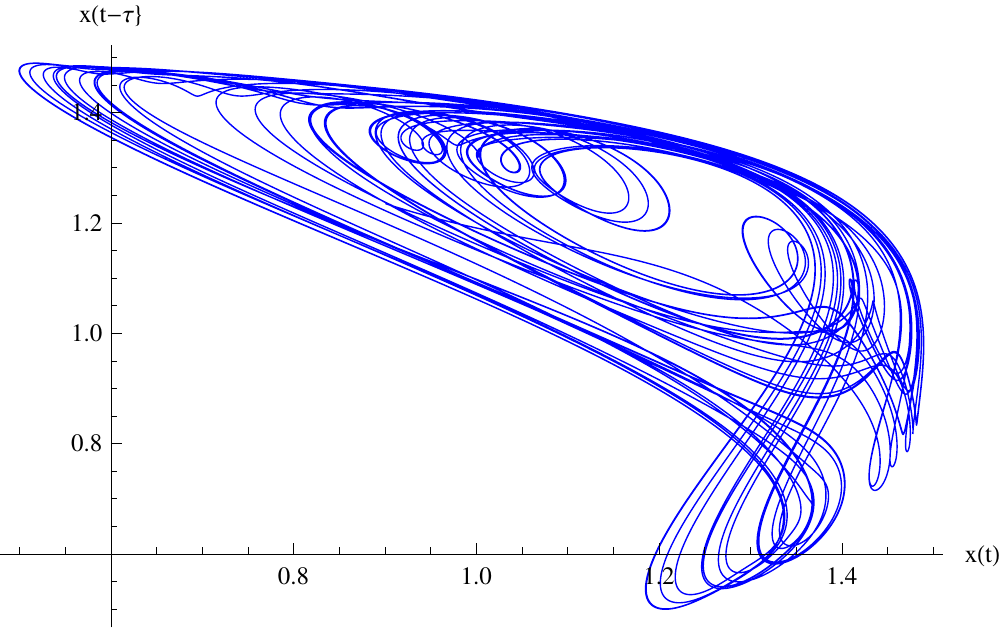}
\caption{ Chaotic attractor for  $\tau=2.3$}\label{5757575675}
\end{figure}
\begin{figure}
 \centering
\includegraphics[width=0.6\textwidth]{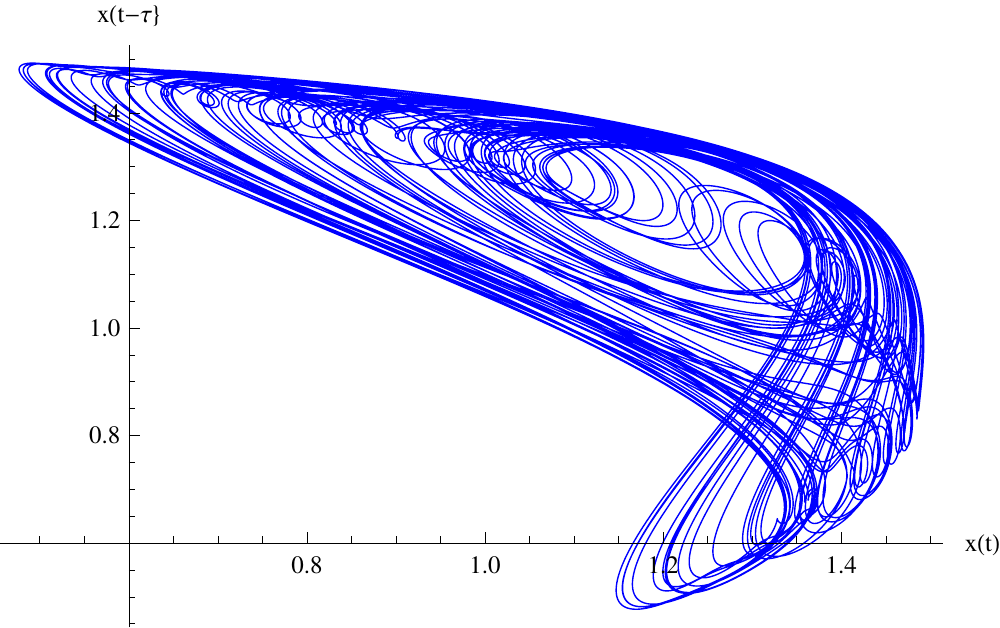}
\caption{ Chaotic attractor for  $\tau=2.5$ }\label{5757575689}
\end{figure}
\section{Conclusion}\label{sec1.4}
In this work, we considered a fractional order delay differential equation
\begin{equation*}
D^\alpha x(t)=\delta x(t-\tau)-\epsilon (x(t-\tau))^3-p(x(t))^2+qx(t). 
\end{equation*}
For some values of parameters, there are three equilibrium points viz. $x_1^*$, $x_2^*$ and $x_3^*$. We provided explicit stability conditions for equilibrium points $x_1^*$ and $x_2^*$. We proposed delay-dependent as well as delay-independent stability conditions. The results are verified by setting particular values to parameters. The key finding is to sketch the stable regions in the $q\delta-$plane which are valid for any $\epsilon>0$ and $p>0.$\\
It is observed that the system shows chaotic oscillations for some range of parameters. We provided the bifurcation diagram and the values of maximum Lyapunov exponents to confirm the chaos in this system.\\
The stability of $x_3^*$ can be done as a future work. 
\section{Acknowledgments}
S. Bhalekar acknowledges the University of Hyderabad for Institute of Eminence-Professional Development Fund (IoE-PDF) by MHRD (F11/9/2019-U3(A)).
D. Gupta thanks University Grants Commission for financial support (Ref.No.:201610026200).
\bibliographystyle{elsarticle-num}
\nocite{*}
\bibliography{bibtexpaper2}
\end{document}